\documentclass{amsart} 
\usepackage[utf8]{inputenc} 
\usepackage[T1]{fontenc}
\usepackage{geometry} 
\usepackage{amsmath,amssymb,amsfonts,mathrsfs}
\usepackage{amsthm,pstricks,enumitem, textcomp, xspace}

\usepackage{graphicx} 
\usepackage{color}
\usepackage{caption}

\usepackage{enumitem} 

\usepackage{todonotes}

\usepackage{hyperref}
\newcommand\myshade{85}
\colorlet{mylinkcolor}{violet}
\colorlet{mycitecolor}{orange}
\colorlet{myurlcolor}{black}

\hypersetup{
  linkcolor  = mylinkcolor!\myshade!black,
  citecolor  = mycitecolor!\myshade!black,
  urlcolor   = myurlcolor!\myshade!black,
  colorlinks = true,
}

\newtheorem{theorem}{Theorem}[section]

\newtheorem{lemma}[theorem]{Lemma}
\newtheorem{definition}[theorem]{Definition}
\let\olddefinition\definition
\renewcommand{\definition}{\olddefinition\normalfont}
\newtheorem{corollary}[theorem]{Corollary}
\newtheorem{proposition}[theorem]{Proposition}

\newtheorem*{claim}{Claim}
\newtheorem{remark}[theorem]{Remark}
\let\oldremark\remark
\renewcommand{\remark}{\oldremark\normalfont}
\newtheorem{observation}[theorem]{Observation}
\let\oldobservation\observation
\renewcommand{\observation}{\oldobservation\normalfont}
\newtheorem{question}[theorem]{Question}
\let\oldquestion\question
\renewcommand{\question}{\oldquestion\normalfont}

\newcommand{\CCC}{CAT(0) cube complex\xspace}

\newcommand{\intc}{intercross\xspace} 
\newcommand{\countc}{countercross\xspace}
\newcommand{\intcs}{intercrosses\xspace}
\newcommand{\countcs}{countercrosses\xspace}
\newcommand{\countcer}{countercrosser\xspace}
\newcommand{\countcers}{countercrossers\xspace}

\newcommand{\N}{\mathbb N}
\newcommand{\Link}{\mathrm{Link}}

\newcommand{\Int}[1]{\mathcal{#1}}
\newcommand{\adjP}{locally parallel\xspace}

\newcommand{\uc}[1]{\tilde{#1}} 

\newcommand{\gp}[1]{#1} 

\newcommand{\mfld}[1]{#1} 

\newcommand{\CC}[1]{\mathbf{#1}} 
\newcommand{\CCv}[1]{\mathbf{#1}} 
\newcommand{\CCc}[1]{\mathbf{#1}} 

\newcommand{\Hyp}[1]{\hat{\mathcal{#1}}} 
\newcommand{\hyp}[1]{\hat{\mathfrak{#1}}} 
\newcommand{\Hypmap}[1]{\hat{#1}_* } 

\newcommand{\ltame}{ ^{\rm{tame}}}
\newcommand{\lwild}{ ^{\rm{wild}}}

\newcommand{\Hs}[1]{\mathcal{#1}} 
\newcommand{\hs}[1]{\mathfrak{#1}} 
\newcommand{\comp}[1]{{#1}^*} 

\newcommand{\simp}[1]{#1} 
\newcommand{\simpe}[1]{#1} 

\newcommand{\usimp}[1]{\tilde{#1}} 
\newcommand{\usimpv}[1]{\tilde{#1}} 

\newcommand{\trk}[1]{#1} 
\newcommand{\ptrn}[1]{{\mathcal{#1}}} 
\newcommand{\utrk}[1]{{\tilde{#1}}} 
\newcommand{\uptrn}[1]{{\tilde{\mathcal{#1}}}} 

\newcommand{\chainofhyps}[4]{\left((\hs{#1}_{#3},\hs{#2}_{#3}),\ldots,(\hs{#1}_{#4},\hs{#2}_{#4})\right)}
\newcommand{\chainoftuples}[4]{\left((\hs{#1}^{1}_{#3},\dots,\hs{#1}^{#2}_{#3}),\ldots,(\hs{#1}^{1}_{#4},\dots,\hyp{#1}^{#2}_{#4})\right)}
\newcommand{\chainofcrosses}[3]{\left(\CCc{#1}_{#2},\ldots,\CCc{#1}_{#3}\right)}
\newcommand{\setofhyps}[4]{\left\{(\hs{#1}_{#3},\hs{#2}_{#3}),\ldots,(\hs{#1}_{#4},\hs{#2}_{#4})\right\}}

\newcommand{\meet}{\wedge} 
\newcommand{\bigmeet}{\bigwedge} 
\newcommand{\join}{\vee} 

\DeclareFontFamily{U}{matha}{\hyphenchar\font45}
\DeclareFontShape{U}{matha}{m}{n}{
      <5> <6> <7> <8> <9> <10> gen * matha
      <10.95> matha10 <12> <14.4> <17.28> <20.74> <24.88> matha12
      }{}
\DeclareSymbolFont{matha}{U}{matha}{m}{n}

\title{Cubical accessibility and bounds on curves on surfaces}

\author{Benjamin Beeker}
\address{Department of Mathematics, Hebrew University, Jerusalem, Israel}
\email{beeker@tx.technion.ac.il}
\urladdr{}

\author{Nir Lazarovich}
\address{Department of Mathematics, ETH Z\"urich, R\"amistrasse 101, 8092 Z\"urich, Switzerland}
\email{nir.lazarovich@math.ethz.ch}

\date{} 
\begin{document}

\begin{abstract}
We bound the size of $d$-dimensional cubulations of finitely presented groups.
We apply this bound to obtain acylindrical accessibility for actions on CAT(0) cube complexes and bounds on curves on surfaces.
\end{abstract}

\maketitle


\section{introduction}
Let $\Sigma$ be a closed surface of genus $g$. It is a well-known fact that the size of a collection of non-homotopic simple closed curves on $\Sigma$ is bounded by $3g-3$.
Such a collection induces an action of $\pi_1(\Sigma)$ on a dual tree.
Sageev \cite{Sag95} showed how a general collection of curves gives rise to an action on a CAT(0) cube complex. This motivates the following definition.
Let $d\in\N$. A collection $\ptrn{S}$ of homotopy classes of essential curves on $\Sigma$ is called a \emph{$d$-pattern} if any pairwise intersecting set of lifts of them to the universal cover $\uc{\Sigma}$ of $\Sigma$ has cardinality at most $d$.
Applying Sageev's construction to a $d$-pattern yields a CAT(0) cube complex of dimension at most $d$.


Thus, one is naturally led to ask the following question.

\begin{question}\label{curves on surfaces?}
Is there a bound $D=D(\Sigma,d)$ on the possible size of $\ptrn{S}$?
\end{question}

Similarly one can define $d$-patterns for collections of subsurfaces in 3-manifolds, and ask a similar question.
Let us note, that for $d=1$, this question was answered by Kneser \cite{Kne29} for collections of subspheres in 3-manifolds, and by Haken \cite{Hak61} and Milnor \cite{Mil62} for general subsurfaces. In \cite{BeLa16}, we answered both question affirmatively for $d=2$.

Dunwoody \cite{Dun85} defined the notion of patterns (which we consider as \emph{$1$-patterns}) on general finite 2-dimensional simplicial complexes. 
As in the case of $1$-patterns on surfaces (and 3-manifolds), $1$-patterns on simplicial complexes give rise to dual trees when lifted to the universal cover.
This fact was used in his paper to study actions of finitely presented groups by introducing \emph{resolutions} and studying their properties. 
In particular, Dunwoody proved that the size of a pattern on a finite 2-dimensional simplicial complex is bounded above by a bound which depends only on the simplicial complex. 
This result is a crucial step in the proof of accessibility, and moreover provides an easy combinatorial proof of the aforementioned bounds on $1$-patterns on surfaces and 3-manifolds.

In \cite{BeLa16}, we introduced the notion of $d$-patterns on 2-dimensional simplicial complexes and resolutions of actions on CAT(0) cube complexes. We will review these definitions in Section \ref{tracks and patterns}.

%

In this paper we extend the main result of \cite{BeLa16} to arbitrary $d$.

\begin{theorem}\label{main result}
	Let $\simp{K}$ be a finite 2-dimensional simplicial complex, and let $d\in\N$. Then there exists a constant $C=C(\simp{K},d)$ such that any $d$-pattern on $\simp{K}$ has at most $C$ parallelism classes of tracks.
\end{theorem}

As corollary we derive the following theorem, which answers Question \ref{curves on surfaces?}.

\begin{theorem}\label{thm: bound on subsurfaces}
	Let $\Sigma$ be a compact surface, and let $d\in\N$. There exists a constant $C=C(K,d)$ such that any $d$-pattern of curves and arcs on $\Sigma$ has at most $C=C(K,d)$ different homotopy classes.
\end{theorem}

Similarly for $3$-manifolds, we have the following.

\begin{theorem}\label{thm: bound on curves}
	Let $\mfld M$ be a compact irreducible, boundary-irreducible 3-manifold, and let $d\in\N$. There exists a constant $C=C(M,d)$, such that if $\ptrn S$ is a collection of non-homotopic, $\pi_1$-injective, 2-sided, embedded subsurfaces, such that the size of a pairwise intersecting collection of lifts to $\uc{\mfld M}$ is at most $d$, then $|\ptrn S|\le C$.
\end{theorem}


For the proofs of Theorem \ref{thm: bound on subsurfaces} and Theorem \ref{thm: bound on curves}  from Theorem \ref{main result}, we refer to Section 5 in~\cite{BeLa16}.

Dunwoody's bound on patterns was extensively used in the literature to study accessibility of group actions on trees.
In this paper, we focus on generalizing acylindrical accessibility for CAT(0) cube complexes.

Let $G$ be a group, $\mathcal{C}$ be a collection of subgroups of $G$ which is closed under conjugation and subgroups, and $k$ be a natural number. 
We say that the group $G$ acts $(k,\mathcal{C})$-\emph{acylindrically} on a tree if the stabilizer of any segment of $k$ edges in the tree belongs to the collection $\mathcal{C}$. 
Similarly one can define $(k,\mathcal{C})$-\emph{acylindricity on hyperplanes} for actions on cube complexes by requiring that the common stabilizer of any chain of $k$ halfspaces belongs to $\mathcal{C}$.
This notion should not be confused with acylindrical actions (and weak acylindrical actions) on metric spaces, see Bowditch \cite{Bow08}, even though the two are related by recent work of Genevois \cite{Gen16}.

In \cite{Sel97}, Sela proved that for any finitely generated group $G$ and $k$, any reduced $(k,\{1\})$-acylindrical action of $G$ on a tree has a bounded quotient, or equivalently, there is a bound on the number of orbits of edges.
In \cite{Del99}, Delzant proved a similar result for finitely presented groups using Dunwoody's bounds on resolutions. He showed that if $G$ is finitely presented and does not split non-trivially over a subgroup in $\mathcal{C}$, then there is a bound that depends on $G$ and $k$ on the number of edge-orbits of $(k,\mathcal{C})$-acylindrical actions of $G$ on a tree.

Since Theorem \ref{main result} applies more generally to cubulations which come from patterns, following Delzant's proof, we are able to prove the following theorem.

\begin{theorem}[Acylindrical accessibility for CAT(0) cube complexes]\label{acylindrical accessibility}
	Let $G$ be a finitely presented group, let $\mathcal{C}$ be a family of subgroups of $G$ which is closed under conjugation, commensurability, and subgroups, and let $d\in\N$. There exists $D=D(d,G)$ such that if $G$ does not act essentially on a $d$-dimensional CAT(0) cube complex with hyperplanes stabilizers in $\mathcal{C}$, then any $(k,\mathcal{C})$-acylindrical on hyperplanes essential action on a $d$-dimensional CAT(0) cube complex has at most $k\cdot D$ hyperplanes.
\end{theorem}


The following Corollary follows from Theorem \ref{acylindrical accessibility} and item \ref{one end implies no CCC over finite} of Proposition \ref{cube complexes to trees}.

\begin{corollary}\label{one ended acyl accessibility}
	Let $G$ be a finitely presented one-ended group, then for all $d$ there exists a constant $C=C(d,G)$ such that every $(k,\mathcal{F})$-acylindrical on hyperplanes action on a $d$-dimensional CAT(0) cube complex has at most $k\cdot C$ hyperplanes, where $\mathcal {F}$ is the collection of all finite subgroups.
\end{corollary}

As an application we prove the following on embeddings of finitely presented one-ended groups into hyperbolic Coxeter groups.

\begin{corollary}
	Let $G$ be a finitely presented one-ended group, and let $d\in\N$. Then there exists $D=D(G,d)$ such that for any embedding of $G$ into a hyperbolic right-angled Coxeter group $W_\Gamma$ on a graph $\Gamma$ with clique number at most $d$, there exists a subgraph $\Gamma'$ with at most $D$ vertices such that the image of $G$ is in a conjugate of the special parabolic subgroup $W_{\Gamma'}\le W_{\Gamma}$.
\end{corollary}

\begin{proof}
	Without loss of generality let $\Gamma$ be such that $G$ does not embed into a conjugate of a proper special subgroup. 
	The embedding induces an action of $G$ on the Davis complex $\CC{X}$ of $W_\Gamma$.
	Each hyperplane of $\CC{X}$ has a corresponding vertex $v$ in $\Gamma$, and the stabilizer of the hyperplane is a conjugate of the special subgroup $W_{\Link(v,\Gamma)}$. 
	By the hyperbolicity of $W_\Gamma$, the stars of any two vertices at distance 2 intersect in a clique. Hence, the common stabilizer of any two adjacent hyperplanes in $A(\Gamma)$ is finite. Thus, the common stabilizer in $G$ is finite. 
	This shows that the action of $G$ on $\CC{X}$ is (2,$\mathcal{F}$)-acylindrical on hyperplanes. 
	By Corollary \ref{one ended acyl accessibility} we obtain the desired conclusion.
\end{proof}

We note that the bounds obtained in Theorem \ref{main result} are probably far from being sharp, since they depend in part on Ramsey's theorem. 
Thus, we did not bother computing them.
However, one may ask what are the effective bounds. 
In particular, even though our bound in Theorem \ref{thm: bound on subsurfaces} depends linearly on the genus of the surface, the question of finding the optimal dependence on $d$ remains open. 

A priori, Question \ref{curves on surfaces?} may appear related to the bounds obtained in  Aougab and Gaster \cite{AoGa15}  or Przytycki \cite{Prz15} on sets of curves with bounded intersections. 
However, we would like to point out that these problem are of fundamentally different nature. 
For example, while there are only finitely many mapping class group orbits of sets of curves with at most $k$ intersections, there are infinitely many orbits of $d$-patterns for any $d\ge 2$. 

\subsection*{Acknowledgements.}

We would like to thank Jonah Gaster for fruitful conversations on Question \ref{curves on surfaces?}. The second author acknowledges the support received by the ETH Zurich Postdoctoral Fellowship Program and the Marie Curie Actions for People COFUND Program.

\section{Preliminaries}

\subsection{CAT(0) cube complexes and pocsets}

We begin by a short survey of definitions concerning CAT(0) cube complexes and pocsets. A reader who is acquainted with the basic terminology can skip this subsection. For further details see, for example, Sageev \cite{Sag12}.

A \emph{cube complex} is a collection of euclidean cubes of various dimensions in which subcubes have been identified isometrically. 

A simplicial complex is \emph{flag} if every $(n+1)$-clique in its 1-skeleton spans a $n$-simplex.
A cube complex is \emph{non-positively curved} (NPC) if the link of every vertex is a flag simplicial complex. It is a \emph{\CCC} if moreover it is simply connected.

A cube complex $\CC{X}$ can be equipped with two natural metrics, the euclidean and the $L^1$-metric. With respect to the former $\CC{X}$ is NPC if and only if it is NPC \`{a} la Gromov (see Gromov \cite{Gro87} or Bridson and Haefliger \cite{BrHa99}). While the latter is more natural to the combinatorial structure of CAT(0) cube complexes described below.

Given a cube $\CCc{C}$ and an edge $\CCc{e}$ of $\CCc{C}$. The midcube of $\CCc{c}$ associated to $\CCc{e}$ is the convex hull of the midpoints of $\CCc{e}$ and the edges parallel to $\CCc{e}$.
A \emph{hyperplane} associated to $\CCc{e}$ is the smallest subset containing the midpoint of $\CCc{e}$ and such that if it contains a midpoint of an edge it contains all the midcubes containing it.
Every hyperplane $\hyp{h}$ in a \CCC $\CC{X}$ separates $\CC{X}$ into exactly two components, see for example Niblo and Reeves \cite{NiRe98}, called the \emph{halfspaces} associated to $\hyp{h}$. A hyperplane can thus also be abstractly viewed as a pair of complementary halfspaces. 
For a \CCC $\CC{X}$ we denote by $\Hyp{H}=\Hyp{H}(X)$ the set of all hyperplanes in $\CC{X}$, and by $\Hs{H}=\Hs{H}(X)$ the set of all halfspaces. For each halfspace $\hs{h}\in \Hs{H}$ we denote by $\comp{\hs{h}}\in\Hs{H}$ its complementary halfspace, and by $\hyp{h}\in\Hyp{H}$ its bounding hyperplane, which we also identify with the pair $\{\hs{h},\comp{\hs{h}}\}$.

If two halfspaces $\hs{h}$ and $\hs{k}$ are such that none of $\hs{h}\cap \hs{k}$, $\comp{\hs{h}}\cap\hs{k}$, $\hs{h}\cap \comp{\hs{k}}$ and $\comp{\hs{h}}\cap \comp{\hs{k}}$ is empty, we write $\hs h \pitchfork \hs k$. 

We adopt Roller's viewpoint of Sageev's construction. Recall from Roller \cite{Rol98} that a \emph{pocset} is a triple $(\Hs{P},\le,\comp{})$ of a poset $(\Hs{P},\le)$ and an order reversing involution $\comp{}:\Hs{P}\to\Hs{P}$ satisfying $\hs{h}\neq\comp{\hs{h}}$ and $\hs{h}$ and $\comp{\hs{h}}$ are incomparable for all $\hs{h}\in\Hs{P}$.

The set of halfspaces $\Hs{H}$ of a \CCC has a natural pocset structure given by inclusion relation, and the complement operation $\comp{}$. Roller's construction starts with a locally finite pocset $(\Hs{P},\le,\comp{})$ of finite width  (see Sageev \cite{Sag12} for definitions) and constructs a \CCC $\CC{X}(\Hs{P})$ such that $(\Hs{H}(X),\subseteq, \comp{})=(\Hs{P},\le,\comp{})$. 



\subsection{Tracks and patterns}\label{tracks and patterns}
The following definition of tracks and patterns is the same as in \cite{BeLa16}. It is a higher dimensional analogue of the definition of tracks and patterns (or ``$1$-patterns'') in Dunwoody \cite{Dun85}. 
As we describe in the next subsection, the $d$-patterns are used to construct $d$-dimensional CAT(0) cube complexes.

\begin{definition}
	A  \emph{drawing} on a $2$-dimensional simplicial complex $\simp{K}$ is a non empty union of simple paths in the faces of $\simp{K}$ such that:
	\begin{enumerate}
		\item on each face there is a finite number of paths,
		\item the two endpoints of each path are in the interior of distinct edges,
		\item the interior of a path is in the interior of a face,
		\item no two paths in a face have a common endpoint,
		\item if a point $x$ on an edge $\simpe{e}$ is an endpoint then in every face containing $\simpe{e}$ there exists a path having $x$ as an endpoint.
	\end{enumerate}
	
	A \emph{pre-track} is a minimal drawing. A pre-track is \emph{self-intersecting} if it contains two intersecting paths.
	
	Denote by $\usimp{K}$ the universal cover.
	
	\begin{itemize}
		\item A pre-track is a \emph{track} if none of its pre-track lifts in $\usimp{K}$ is self-intersecting.
		\item A \emph{pattern} is a set of tracks whose union is a drawing.
		\item A \emph{$d$-pattern} is a pattern such that the size of any collection of lifts of its tracks in $\usimp{K}$ that pairwise intersect is at most $d$.
	\end{itemize}
\end{definition}

We will sometimes view a pattern as the unions of its tracks in $\simp{K}$.

\subsection{The pocset structures associated to a pattern}\label{pocset} 

Let $\uptrn{P}$ be a pattern on a simply connected 2-simplex $\usimp{K}$.
For each track $\utrk{t}$ of $\uptrn{P}$, the set $\usimp{K}^0$ is naturally split by $\utrk{t}$ in two components  $\hs{h}_{\utrk{t}}$ and $\comp{\hs{h}_{\utrk{t}}}$ (see Dunwoody \cite{Dun85}). 
We call these components the \emph{halfspaces defined by $\utrk{t}$}, and the collection of all halfspaces is denoted by $\Hs{H}=\Hs{H}(\ptrn{P})$. 
This collection forms a locally finite pocset with respect to inclusion and complement operation $\comp{}$. 
If moreover $\uptrn{P}$ is a $d$-pattern, then $\Hs{H}$ has finite width. We denote by $\CC{X}=X(\Hs{H})$ the \CCC constructed from the pocset $\Hs{H}$. 
Note that the dimension of $\CC{X}$ is at most $d$. 

Note that the map $\Hypmap{\phi}$ sending $\utrk{t}\in\ptrn{P}$ to the hyperplane $\{\hs{h}_{\utrk{t}},\comp{\hs{h}_{\utrk{t}}}\} \in \Hyp{H}=\Hyp{H}(\CC{X})$ is not injective. 

\begin{definition}[parallelism]
	Two tracks of a pattern are \emph{parallel} if they define the same halfspaces. In other words if they have the same image under the map $\Hypmap{\phi}$.
\end{definition}

\subsection{Resolutions} \label{resolutions}
Let $\gp{G}$ be a finitely presented group and  $\simp{K}$ be a finite triangle complex such that $\gp{G} = \pi_1(\simp{K})$.
Given an action of $G$ on $\CC{X}$ a $d$-dimensional \CCC, we can associate a (non canonical) $d$-pattern $\ptrn{P}$ on $\simp K$ in the following way.

First build $\varphi$ a $\gp{G}$-equivariant map from $\usimp{K}$ the universal cover of $\simp{K}$ to $\CC{X}$ by arbitrarily assigning an image for a representative of each orbit of vertices of $\usimp{K}$, and then extending $\gp{G}$-equivariantly to all vertices, edges and triangles. 
The pullback of the hyperplanes of $\CC{X}$ is a $\gp{G}$-equivariant pattern on $\usimp{K}$ that induces a pattern $\ptrn{P}$ on $\simp{K}$.

As describe previously, the pattern $\ptrn{P}$ is associated to a pocset structure and a \CCC $\CC{X}'$ called a \emph{resolution} of $\CC{X}$. This resolution is naturally endowed with a $\gp{G}$-equivariant map to $\CC{X}$.

 Proofs and more properties of resolutions can be found in \cite{BeLa16,BeLa16b}.

\subsection{Intervals, crosses, meets and joins}

	Let $\CC{X}$ be a CAT(0) cube complex, and let $\CCv{x},\CCv{y}$ be two vertices in $\CC{X}$. The \emph{interval}  $\Int{I}=[\CCv{x},\CCv{y}]$ spanned by $\CCv{x}$ and $\CCv{y}$ is the poset of all halfspaces satisfying $\CCv{x}\in\hs{h}$ and $\CCv{y}\in\comp{\hs{h}}$.

\begin{remark}
	We remark that usually the interval is defined to be the $L^1$ convex hull of $\CCv{x}$ and $\CCv{y}$.
	For an interval $\Int{I}$ the set $\{\hs{h},\comp{\hs{h}} | \hs{h}\in\Int{I}\}$ is naturally a pocset.
	The associated cube complex is isomorphic to the $L^1$ convex hull of $\CCv{x}$ and $\CCv{y}$ in $\CC{X}$.
\end{remark}

	A \emph{cross} in a cube complex $\CC{X}$ is a collection of pairwise crossing hyperplanes.
	Similarly, a \emph{cross} in an interval $\Int{I}$ is a pairwise incomparable collection of halfspaces.
	The dimension of a cross is its size.

	Let $\Int{I}$ be an interval. On the set of crosses of $\Int{I}$ we define the \emph{meet} (denoted $\meet$) and  \emph{join} (denoted $\join$) operations by:
	\begin{itemize}
		\item $\CCc{C} \meet \CCc{C'} = \left\{ \hs{h} \in \CCc{C} \cup \CCc{C}' | \nexists \hs{k} \in \CCc{C} \cup \CCc{C}',~ \hs{k} < \hs{h} \right\}$.
		\item $\CCc{C} \join \CCc{C'} = \left\{ \hs{h} \in \CCc{C} \cup \CCc{C}' |\nexists \hs{k} \in \CCc{C} \cup \CCc{C}',~ \hs{k} > \hs{h} \right\}$.
	\end{itemize}

By definition, the meet and join are again crosses in the interval $\Int{I}$.

\begin{observation}
	With respect to these operations the set of crosses of $\Int{I}$ form a (distributive) lattice.
	
	Moreover $\#\CCc{C}+ \#\CCc{C'} \leq \#\left(\CCc{C} \meet \CCc{C'}\right) + \#\left( \CCc{C} \join \CCc{C'}\right)$, and  $\CCc{C}\cup \CCc{C'} = \left(\CCc{C} \meet \CCc{C'}\right) \cup \left( \CCc{C} \join \CCc{C'}\right)$.
	
\end{observation}
\section{Intercrosses and Countercrosses}

Let $\Int{I}$ be an interval. Let $\hs{h}<\hs{k}$ be two halfspaces of $\Int{I}$. 
We say that $\hs{h}$ and $\hs{k}$ are \emph{adjacent} if there is no halfspace $\hs{t}$ such that $\hs{h}<\hs{t}<\hs{k}$. An \emph{\intc} with respect to $\hs{h}<\hs{k}$ is a (non empty) cross $\CCc{C} \subset \Int{I}$ decomposed as two disjoints sets $\CCc{C} = \CCc{H} \cup \CCc{K}$ such that
\begin{itemize}
\item  every element of $\CCc{H}$ is transverse to $\hs{h}$,
\item  every element of $\CCc{K}$ is transverse to $\hs{k}$ and disjoint from $\hs{h}$.
\end{itemize}

Let $\hs{h}<\hs{k}$ be two halfspaces in $\Int{I}$ and let $\CCc{C} = \CCc{H} \cup \CCc{K}$ be an \intc for $\hs{h}$ and $\hs{k}$. A \emph{\countc} is a cross $\CCc{C}' \subset \Int{I}$ such that:
\begin{itemize}
\item $\# \CCc{C} < \# \CCc{C}'$,
\item if $\CCc{K} \neq \emptyset$, there exist elements $\hs{k}'\in\CCc{C}'$ and $\hs{k}\in\CCc{K}$ such that $\hs{k'}\le\hs{k}$,
\item if $\CCc{K} = \emptyset$, then $\hs h\in\CCc{C}'$,
\item there exists an element $\hs{k}'\in\CCc{C}'$ such that $\hs{k}'\ge\hs h$.
\end{itemize}



Given two halfspaces $\hs{h}<\hs{k}$, we say that $\hs{k}$ is \emph{locally parallel} to $\hs{h}$ if they are adjacent and for any \intc $\CCc{C}$ between them and any other adjacent pair $\hs{h}<\hs{k'}$ admits an \intc of dimension greater or equal to the one of $\CCc{C}$.

We emphasize the fact that these definitions are oriented. In particular, if $\hs{k}$ is locally parallel to $\hs{h}$ in $\Int{I}$, it does not imply that $\comp{\hs{h}}$ is locally parallel to $\comp{\hs{k}}$ with respect to the inverse orientation of $\Int{I}$.



\begin{lemma}\label{lemma0}
Let $\Int{I}$ be an interval, and let $\hs{h}$ be a non-maximal halfspace. Then there exists an adjacent halfspace $\hs{k}>\hs h$, for which any \intc admits a \countc.

\end{lemma}

We call such a halfspace a \emph{\countcer}.

\begin{proof}

Let $\Hs{K} = \left\{\hs{k}_1,\dots, \hs{k}_n\right\}$ be the set of halfspaces adjacent to and above $\hs{h}$.
If one element of $\Hs{K}$ does not share an \intc with $\hs{h}$ then it verifies the Lemma.

Otherwise for each $i$, let $\CCc{C}_i = \CCc{H}_i \cup \CCc{K}_i$ be an \intc for the pair $(\hs{h},\hs{k}_i)$. To prove the lemma we need to show that one of these \intcs admits a \countc.

Notice that if some $\CCc{K}_i$ is empty, then $\CCc{C}_i\cup \left\{\hs h\right\}$ is a \countc for $\CCc{C}_i$. Similarly, if $\CCc{K}_i$ is not empty and there is no halfspace in $\CCc{H}_i$ which is strictly below $\hs{k}_i$, then $\CCc{C}_i\cup \left\{\hs k_i\right\}$ is a \countc for $\CCc{C}_i$. We thus can assume that for all $i$ the set $\CCc{K}_i$ is non-empty and there exists $\hs{s}\in\CCc{H}_i$ such that $\hs{s}<\hs{k}_i$ (and in particular, $\CCc{H}_i$ is non-empty).

Notice that for any $\hs{t} \in \bigcup_{\hs{k_i} \in \Hs{K}}\CCc{K_i}$ there exists $j$ such that $\hs{k}_j \leq \hs{t}$, and therefore for some $\hs{s} \in \CCc{H}_j$, we have $\hs{s} < \hs{t}$. This implies that $\left(\bigmeet_{\hs{k_i} \in \Hs{K}}\CCc{C}_i\right)\cap\left(\bigcup_{\hs{k_i} \in \Hs{K}}\CCc{K_i}\right) = \emptyset$.

 Let $\Hs{K}'$ be a minimal subset of $\Hs{K}$ such that $\bigmeet_{\hs{k}_i \in \Hs{K}'} \CCc{C}_i \cap \bigcup_{\hs{k}_i \in \Hs{K}'} \CCc{K}_i = \emptyset$.
  \begin{claim}
  For any proper non-empty subset $\Hs{K}'' \subset \Hs{K}'$, there exists $\hs{k}_j \in \Hs{K}'\setminus\Hs{K}''$ such that \[\CCc{K}_j \cap \left( \bigmeet_{\hs{k}_i\in \Hs{K}''} \CCc{C}_i \join \CCc{C}_j\right) \neq \emptyset.\]
  \end{claim}
  \begin{proof}By contradiction, assume that for all $\hs{k}_j \in  \Hs{K}'\setminus \Hs{K}''$, we have $\CCc{K}_j \subset  \bigmeet_{\hs{k}_i\in \Hs{K}''} \CCc{C}_i \meet \CCc{C}_j$. 
  But this implies that
  \[ \bigmeet_{\hs{k}_j\in \Hs{K}'\setminus \Hs{K}''} \CCc{C}_j \cap \bigcup_{\hs{k}_j \in  \Hs{K}'\setminus \Hs{K}''}\CCc{K}_j \subseteq \bigmeet_{\hs{k}_i\in \Hs{K}''} \CCc{C}_i \meet \bigmeet_{\hs{k}_j\in \Hs{K}'\setminus \Hs{K}''} \CCc{C}_j=\bigmeet_{\hs{k}_i \in \Hs{K}'} \CCc{C}_i.\] 
  But since $\Hs{K}'$ verifies $\bigmeet_{\hs{k}_i \in \Hs{K}'} \CCc{C}_i \cap \bigcup_{\hs{k}_i \in \Hs{K}'} \CCc{K}_i = \emptyset$, it would imply that  $\Hs{K}'\setminus \Hs{K}''$ verifies 
  \[\bigmeet_{\hs{k}_j\in \Hs{K}'\setminus \Hs{K}''} \CCc{C}_j \cap \bigcup_{\hs{k}_j \in  \Hs{K}'\setminus \Hs{K}''}\CCc{K}_j = \emptyset,\] which contradicts the minimality of $\Hs{K}'$.
\end{proof}
Let us now construct a \countc for some element of $\Hs{K}'$.
Choose some $\hs{k}_{i_1}\in\Hs{K}'$, and set $\Hs{K}''_1 = \{\hs{k}_{i_1}\}$ and $\CCc{D}_1 = \CCc{C}_{i_1}$.
We will construct subsets $\Hs{K}''_i\subset\Hs{K}'$, of size $i$, and crosses $\CCc{D}_i$ inductively, so that they satisfy:
\begin{itemize}
\item $\CCc{D}_i= \bigmeet _{\hs{k}_j\in\Hs{K}''_i} \CCc{C}_j$,
\item $\#\CCc{D}_i\ge\#\CCc{D}_{i-1}$, and
\item $\Hs{K}''_{i-1}$ is strictly contained in $\Hs{K}''_i$.
\end{itemize}

We construct $\CCc{D}_i$ from $\CCc{D}_{i-1}$ in the following way. By the claim there exists $\hs{k}_j\in\Hs{K}'\setminus\Hs{K}''_{i-1}$ such that an element of $\CCc{K}_j$ belongs to $\CCc{D}_{i-1} \join \CCc{C}_j$. If $\#(\CCc{D}_{i-1} \join \CCc{C}_j)>\#\CCc{C}_j$, then $\CCc{D}_{i-1} \join \CCc{C}_j$ is a \countc for $\CCc{C}_j$, and we are done. 
Otherwise $\#(\CCc{D}_{i-1} \meet \CCc{C}_j)\geq \# \CCc{D}_{i-1}$, and we can define $\Hs{K}''_i = \Hs{K}''_{i-1}\cup\{\hs{k}_j\}$ and accordingly $\CCc{D}_i = \CCc{D}_{i-1} \meet \CCc{C}_j$.
 
If we did not find a \countc in the process, we end up (after $p=\#\Hs{K}'$ steps) with $\Hs{K}''_p=\Hs{K}'$ and $\CCc{D}_{p} = \bigmeet_{\hs{k_i} \in \Hs{K}'}\CCc{C}_i$. But since $\CCc{D}_p \cap \CCc{K}_{i_1}=\emptyset$,  $\hs{h}$ is transverse to every element of  $\CCc{D}_{p}$ and therefore $\CCc{D}_{p}\cup \{\hs{h}\}$ (which satisfies $\#(\CCc{D}_{p}\cup \{\hs{h}\})> \#\CCc{D}_p\ge\ldots\ge\#\CCc{D}_1=\#\CCc{C}_{i_1}$) is a \countc for  $\hs{k}_{i_1}$.
\end{proof}


\section{Reductions for sequences}

The goal of this section is to describe the various reductions we will use when considering sequences of pairs of halfspaces and crosses. We assume throughout that the intervals involved have dimension at most $d$.


 A \emph{chain} of halfspaces is a sequence of halfspaces $(\hs{h}_1,\ldots,\hs{h}_n)$ such that either $\hs{h}_1<\hs{h}_2\ldots<\hs{h}_n$, $\hs{h}_1>\hs{h}_2\ldots>\hs{h}_n$ or $\hs{h}_1=\hs{h}_2=\ldots=\hs{h}_n$. We say that the chain is \emph{increasing}, \emph{decreasing}, or \emph{constant} respectively.
 
 A \emph{chain of $p$-tuples} is a sequence of $p$-tuples of halfspaces $\chainoftuples{t}{p}{1}{n}$ such that for all $1\leq i\leq p$, the sequence $\hs{t}^i_1,\ldots,\hs{t}^i_n$ is a chain.
 
 A sequence of crosses $\chainofcrosses{C}{1}{N}$ is \emph{regularly ordered} if all crosses have same dimension $p$ and if there exists a chain of $p$-tuples $\chainoftuples{t}{p}{1}{n}$, such that $C_i = \left\{\hs{t}^1_i,\dots \hs{t}^p_i\right\}$. It is \emph{regularly increasing}, if non of the chains are decreasing.
 A \emph{subchain} of a regularly ordered sequence of crosses is one of the chains $\hs{t}^j_1,\dots, \hs{t}^j_n$.

\begin{observation}
	\label{Ramsey}
	For all $n$ there exists $R(n,d)$ such that any sequence of $R(n,d)$ (not necessarily distinct) halfspaces $\hs{h}_1,\ldots,\hs{h}_{R(n,d)}$ contains a subsequence which is a chain of length $n$.
\qed
\end{observation}

By applying Observation \ref{Ramsey} several times one can deduce the following lemmas.

\begin{lemma}\label{Ramseypowerup}
	For every $n$ and every $p$ there exists $N=N(n,d,p)$ such that every sequence $\chainoftuples{t}{p}{1}{N}$ of $N$  $p$-tuples of halfspaces. Then there exist subsequence of $n$ $p$-tuples $\chainoftuples{t}{p}{i_1}{i_n}$ which is a chain of $p$-tuples. \qed
\end{lemma}

The following lemma, which follows from Ramsey's Theorem, shows that for an increasing chain of pairs one can reduce to one of two extreme scenarios:
\begin{itemize}
	\item a \emph{staircase} is an increasing chain of pairs $\chainofhyps{h}{k}{1}{n}$ such that $\hs k_i > \hs h_j$ for all $i\geq j$ and $\hs k_i \pitchfork \hs h_j$ for all $i<j$,
	
	\item a \emph{ladder} is an increasing chain of pairs $\chainofhyps{h}{k}{1}{n}$ such that $\hs h_i < \hs k_i$ for all $i$ and $\hs{k}_i < \hs h_{i+1}$ for all $i<n$.
\end{itemize}

\begin{lemma}\label{reductionforpairs}
	For every $n$ there exists $N=N(n,d)$ such that for every set of $N$ distinct adjacent pairs of halfspaces $\setofhyps{h}{k}{1}{N}$  there exists a increasing chain sequence of $n$ pairs $\chainofhyps{h}{k}{i_1}{i_n}$ which is either a staircase or a ladder.
\end{lemma}

\begin{proof}
	By applying Lemma \ref{Ramseypowerup} we may assume that  $\chainofhyps{h}{k}{1}{N}$ is a chain, and by reordering we may assume that that both subchain are increasing (notice that $d$ bounds the number of distinct adjacent halfspaces to a given halfspace $\hs{h}$).
	Let us consider the graph whose vertices are the pairs $\chainofhyps{h}{k}{1}{N}$, and whose edges are the pairs $(\hs{h}_i,\hs{k}_i)$ and $(\hs{h}_j,\hs{k}_j)$ ($i<j$) such that $k_i$ crosses $h_j$.
	By Ramsey's theorem there exists $N$ such that either there exists a $n$-clique or a $n$-independent set, these correspond to the staircase and ladder scenarios.
\end{proof}

	Given a chain of pairs $\chainofhyps{h}{k}{1}{n}$ and a chain of halfspaces $\hs{t}_1,\ldots,\hs{t}_{n-1}$. we say that the chain of halfspaces is \emph{tame} with respect to the chain of pairs if for all $1\leq i<n$, we have $\hs{t}_i<\hs{h}_{i+1}$. The chain of halfspaces is \emph{wild} if $\hs{t}_i \pitchfork \hs{h}_{j}$ for all $j>i$.

	Let $\chainofhyps{h}{k}{1}{n}$ be a chain of pairs  and $\chainofcrosses{C}{1}{n}$ be a regularly ordered chain of crosses. The chain of crosses is \emph{weakly tame} if for each halfspace $t_i \in C_i$ we have $\hs t_i\not > \hs h_{i+1}$. It is \emph{tame} if one of its subchains is tame. It is \emph{$\CCc{K}$-tame} if for all subchain of halfspaces  $\hs{t}_1,\ldots,\hs{t}_{n-1}$, either the chain is tame or for all $i$, $\hs{t}_i \pitchfork \hs{h}_i$.

Note that in the case of crosses, tame and $\CCc{K}$-tame imply weakly tame. However since $\CCc{K}$ may be empty, $\CCc{K}$-tame does not imply tame.

\begin{lemma}\label{reductionforcrosses}
	For all $n$ there exists $N=N(n,d)$ such that for every sequence  $\chainofcrosses{C}{1}{N}$ of crosses of dimension $p$ there exists a regularly increasing sequence $\chainofcrosses{D}{i_1}{i_n}$ of crosses of dimension $p$ such that if $\chainofcrosses{C}{1}{N}$ have any of the following properties
	\begin{itemize}
		\item tame;
		\item \intcs;
		\item having a subchain $(\hs{t}_1, \dots \hs{t}_n)$ such that $\hs{t}_i \geq \hs{h}_i$;
	\end{itemize}
	with respect to $\chainofhyps{h}{k}{1}{N}$, then $\chainofcrosses{D}{i_1}{i_n}$ have the same properties with respect to $\chainofhyps{h}{k}{i_1}{i_n}$.
	If moreover $\chainofcrosses{C}{1}{N}$ are weakly tame then one can choose  $\chainofcrosses{D}{i_1}{i_n}$ so that every subchain $(\hs{t}_{i_1}, \dots \hs{t}_{i_n})$ of $\chainofcrosses{D}{i_1}{i_n}$ is either tame of wild.
\end{lemma}

\begin{proof}
	By applying Lemma \ref{Ramseypowerup} we may pass to a subsequence of $N'$ crosses which is regularly ordered. By abuse of notation we will assume that $\chainofcrosses{C}{1}{N'}$ are regularly ordered.
	Let us consider the crosses
	$$\CCc{D}_i= \{\hs{t}_i | \{\hs{t}_j\}_j \text{ is non-decreasing}\} \cup \{\hs{t}_{n-i} | \{\hs{t}_j\}_j\text{ is decreasing}\}.$$
	The sequence $\chainofcrosses{D}{1}{N'}$ is a regularly increasing sequence of crosses. Moreover, each of the three properties in the lemma pass on to $\chainofcrosses{D}{1}{N'}$.
	
	As in the proof of Lemma \ref{reductionforpairs}, an application of Ramsey's theorem shows that for $N'$ big enough, we can pass to a subsequence of $n$  crosses, which by abuse of notation we will denote again by $\chainofcrosses{D}{1}{n}$ such that every subchain $(\hs{t}_1, \dots \hs{t}_n)$ of $\chainofcrosses{D}{1}{n}$ is either tame or wild. 
	Since the three properties pass to subsequences they remain true for $\chainofcrosses{D}{1}{n}$. 
	Note that weak tameness is needed to insure that each $\hs t_i$ is below or transverse to $\hs h_{i+1}$, and thus is not above all $\hs h_{j}$.
\end{proof}

\begin{lemma}\label{vertical horizontal trick}
	Let $\chainofhyps{h}{k}{1}{n}$ be a staircase, and let $\chainofcrosses{C}{1}{n}$ be a tame regularly increasing sequence of crosses of dimension $p$ such that any subchain is either tame or wild. Assume that each cross $\CC{C}_i$ contains a halfspace $\hs{s}_i$ such that $\hs{h}_i\leq\hs{s}_i$. Then there exists a regularly increasing  sequence of crosses $\chainofcrosses{D}{1}{n'}$ of dimension $p$ which are tame and $\CCc{K}$-tame \intc, with tame or wild subchains with respect to the chain of pairs $\chainofhyps{h}{k}{2}{2n'}$ of even indices where $n' =\lfloor\frac{ n-1}{2} \rfloor$.
\end{lemma}
 
\begin{proof}
	Let $\CCc{C}\lwild$ (resp. $\CCc{C}\ltame$) be the set of all wild (resp. tame) halfspaces in $\CCc{C}$. Then the sequence $\chainofcrosses{D}{1}{n-1}$ of crosses which are defined by $\CCc{D}_i = \CCc{C}\lwild_i \cup \CCc{C}\ltame_{i+2}$ are \intcs for $(\hs{h}_{i+1},\hs{k}_{i+1})$. 
	The set $\CCc{D}_i$ is a cross because an element in $\CCc{C}\ltame_{i+2}$ cannot be strictly below an element of $\CCc{C}_{i}$ by the regular increasing order on $\CCc{C}_i$, and it cannot be strictly above an element of $\CCc{C}\lwild_i$ since it is below $\hs{h}_{i+3}$ and every element of $\CCc{C}\lwild_i$ crosses $\hs{h}_{i+3}$. 
	Moreover elements of $\CCc{C}\lwild_i$ intersect $\hs{h}_{i+1}$, elements  of $\CCc{C}\ltame_{i+2}$ are not smaller or equal to $\hs h_{i+1}$ since they intersect $\hs{s}_{i+2}$, and cannot be above $\hs{k}_{i+1}$ since  $\hs{k}_{i+1} \pitchfork \hs{h}_{i+3}$. Therefore $\CCc{D}_i$ is an \intc.
	
	Since $\chainofcrosses{C}{1}{n}$ are tame, the chain of crosses with odd indices $\{\CCc{D}_1, \CCc{D}_3, \dots\}$ is tame with respect to the subsequence of $\{(\hs{h}_{2},\hs{k}_{2}),(\hs{h}_{4},\hs{k}_{4}),\ldots\}$ of even indexed pairs. 
	It is also $\CCc{K}$-tame because the only halfspaces that do not intersect $\hs{h}_i$ are coming from $\CCc{C}\ltame_{i}$.
\end{proof}



\section{Bounds on locally parallel pairs of halfspaces}


\begin{lemma}\label{lemma1staircase}
Given an interval $\Int I$. There exists a constant $C$ depending only on the dimension such that at most $C$ pairs of \adjP halfspaces can form a staircase.
\end{lemma}

\begin{corollary}\label{lemma1}
Given an interval $\Int{I}$ and a point $\CCv{m}$ of $\Int{I}$. There exists a constant $C$ depending only on the dimension such that at most $C$ pairs of \adjP halfspaces are separated by $\CCv{m}$.
\end{corollary}

\begin{proof}[Proof of Corollary \ref{lemma1}]
If there is no bound, from Lemma \ref{reductionforpairs}, we can assume, that the pairs of \adjP halfspaces form a staircase or a ladder. But since each of the pairs is separated by $\CCc{m}$ it has to be a staircase. Lemma \ref{lemma1staircase} concludes.
\end{proof}

\begin{proof}[Proof of Lemma \ref{lemma1staircase}]
By contradiction, assume that for any $C$, there exists an interval $\Int I$ and a staircase of \adjP pairs $(\hs{h}_1, \hs{k}_1)\dots (\hs{h}_C, \hs{k}_C)$. 

For each $1\le i<C$, let $\hs{l}_i > \hs{h}_i$ and $\hs{o}_i > \hs{h}_i$ be halfspaces adjacent to $\hs{h}_i$ such that $\hs{l}_i \leq \hs{h}_{i+1}$ and $\hs{o}_i$ is a \countcer (see Lemma \ref{lemma0}). 
Moreover if $\hs{o}_i \leq \hs{h}_{i+1}$, we assume $\hs{o}_i=\hs{l}_i$. Let $\CCc{C}_{\hs{l},i}$ be an \intc for the pair $(\hs{h}_i, \hs{l}_i)$ of maximal dimension.

Since $\hs{l}_i\le\hs{h}_{i+1}$ for $1\le i<C$, the intercrosses $\left\{\CCc{C}_{\hs{l},1},\ldots,\CCc{C}_{\hs{l},C-1}\right\}$ are weakly tame with respect to $\chainofhyps{h}{k}{1}{C}$. Hence by Lemma \ref{reductionforcrosses} we may assume the following.
\begin{enumerate}
\item \label{l is o} The sequence of halfspaces $\left\{\hs{o}_i\right\}$ is either tame or wild with respect to the chain of pairs $\chainofhyps{h}{k}{1}{C}$. 
In particular either $\hs{o}_i = \hs{l}_i$ for all $i$ or $\hs{o}_i \neq  \hs{l}_i$ for all $i$.
\item The dimension of the $\CCc{C}_{\hs{l},i}$ is a constant that we denote $p$.
\item \label{reginccubes} The $\CCc{C}_{\hs{l},i}$ are regularly increasing.
\item Either $\CCc{C}_{\hs{l},i}$ contains a halfspace $\geq \hs{h}_i$, or $\CCc{C}_{\hs{l},i}\cup \hs{h}_i$ is a cross (which trivially contains a halfspace $\geq \hs{h}_i$).
\item \label{tameorwildcrosses} Every subchain of halfspaces of the chain of crosses is either tame or wild.
\end{enumerate}

Note that since  $\CCc{C}_{\hs{l},i}$ is a \intc of maximal dimension for $\hs{h}_i$ and $\hs{l}$, by definition of \adjP the pair $(\hs{h}_i,\hs{k}_i)$ share no \intc of dimension $>p$.

If the $\CCc{C}_{\hs{l},i}$ are wild, then $\CCc{C}_{\hs{l},i} \cup \left\{\hs{h}_{i+2}\right\}$ is a \intc for the pair $(\hs{h}_{i+1}, \hs{k}_{i+1})$ since $\CCc{C}_{\hs{l},i}$ is transverse to $\hs{h}_{i+1}$ and $\hs{h}_{i+2}$ is transverse to $\hs{k}_{i+1}$. But this is a contradiction as the dimension of this \intc is $p+1$.

So the crosses $\CCc{C}_{\hs{l},i}$ are tame.
We first build a tame \intc for the pair $(\hs{h}_i, \hs{o}_i)$.

If $\hs{o}_i = \hs{l}_i$ for all $i$ then $\CCc{C}_{\hs{l},i}$ is the cross that we want. Otherwise the halfspaces $\hs{o}_i$ are wild, i.e, the chain of pairs $\chainofhyps{h}{o}{1}{C}$ form a staircase. 

We can apply Lemma \ref{vertical horizontal trick}, to obtain a regularly increasing sequence of tame and $\CCc{K}$-tame \intc $\CCc{C}_{\hs{o},i}$ for the pairs $(\hs{h}_i,\hs{o}_i)$ for the even indices $1<i<C$.

As $\hs{o}_i$ are \countcers, we can produce \countcs $\CCc{C}'_{\hs{o},i}$ of dimension $p+1$ for the crosses $\CCc{C}_{\hs{o},i}$. Tameness and $\CCc{K}$-tameness of $\CCc{C}_{\hs{o},i}$ and the definition of the \countc imply that $\CCc{C}'_{\hs{o},i}$ is tame and has an element above $\hs{h}_i$.

Using Lemma \ref{reductionforcrosses}, we can assume that the crosses $\CCc{C}'_{\hs{o},i}$ are regularly increasing and that every subchain of $\left\{\CCc{C}'_{\hs{o},i}\right\}$ is tame or wild.
We can then apply Lemma \ref{vertical horizontal trick}, to get \intcs of dimension $p+1$ for the pairs $(\hs{h}_i,\hs{k}_i)$ when $4|i$, which is a contradiction.
\end{proof}

\begin{lemma}\label{lemmaforlemma2}
	For all $n$ there exists $N=N(n,d)$ such that for every ladder of adjacent halfspaces $\chainofhyps{h}{k}{1}{N}$ and a regularly increasing sequence of $d$-dimensional \intcs $\chainofcrosses{C}{1}{N}$ there is a sequence  $\chainofcrosses{D}{i_1}{i_n}$ of $n$ regularly ordered crosses such that for every subchain $\hs{t}_{i_1},\ldots,\hs{t}_{i_n}$, either for all $1\le r\le n$ the halfspace $\hs{t}_{i_r}$ crosses $\hs{h}_{i_j}$ for all $j$ (in which case we call it \emph{unbounded}), or for all $1\le r\le n$ $\hs{t}_{i_r}$ is between $\hs{h}_{i_{r-1}}$ and $\hs{h}_{i_{r+1}}$ (in which case we call it \emph{bounded}).
\end{lemma}	

\begin{proof}
	Since $\chainofhyps{h}{k}{1}{N}$ is a ladder and $\chainofcrosses{C}{1}{N}$ are \intcs, it follows that $\chainofcrosses{C}{1}{N}$ are weakly tame with respect to $\chainofhyps{h}{k}{1}{N}$. Lemma \ref{reductionforcrosses} applied twice for the two orientations of the interval, gives a subsequence of crosses, which by abuse of notation we denote by $\chainofcrosses{C}{0}{n+1}$, in which every halfspaces is one of the 4 possible options of being tame/wild in the two directions. 
	Let us denote the partition of each cross into the 4 categories by $\CCc{C}_i^{ut,dt},\CCc{C}_i^{uw,dt},\CCc{C}_i^{ut,dw},\CCc{C}_i^{uw,dw}$, where the letters stand for \underline{u}p, \underline{d}own, \underline{t}ame and \underline{w}ild.
	For $i=1,\ldots,n$ form the crosses $\CCc{D}_i$ by \[\CCc{D}_i = \CCc{C}_i^{ut,dt}\cup\CCc{C}_0^{uw,dt}\cup\CCc{C}_{n+2}^{ut,dw}\cup\CCc{C}_i^{uw,dw}.\]
	It is easy to verify that the sets $\CCc{D}_i$  are \intcs and that they have the desired property with respect to $\chainofhyps{h}{k}{1}{n}$.
\end{proof}

\begin{lemma}\label{lemma2}
	Let $\CCv{x},\CCv{y}_1,\CCv{y}_2$ be three vertices, let $\CCv{m}$ be their median, and let $\Int{I}_i$, $i=1,2$, be the interval spanned between $\CCv{x}$ and $\CCv{y}_i$. There exists a constant $C$ depending only on the dimension such that at most $C$ pairs of adjacent halfspaces which separate $\CCv{x}$ and $\CCv{m}$, are \adjP in $\Int{I}_1$ but not in $\Int{I}_2$.
	
	The same statement is also true for the intervals $\Int{I}'_i =[\CCv{y}_i, \CCv{x}]$.
\end{lemma}

\begin{proof}
	By contradiction, let $(\hs{h}_1, \hs{k}_1)\dots (\hs{h}_C, \hs{k}_C)$ be such pairs. By Lemma \ref{reductionforpairs} we can assume that it forms a staircase or a ladder, and by 
	Lemma \ref{lemma1staircase} we can assume that it is a ladder. 
	
	Let $\CCc{C}_i$ be the \intc of maximal dimension in $(\hs{h}_i,\hs{k}_i)$ in the interval $\Int{I}_2$. 
	By assumption, for every $i$ there exists a locally parallel halfspace $\hs{t}_i$ for $\hs{h}_i$ such that all the \intcs in $(\hs{h}_i,\hs{t}_i)$ in $\Int{I}_2$ have strictly smaller dimension than that of $\CCc{C}_i$.
	By Lemmas \ref{reductionforpairs} and \ref{lemma1staircase} we may assume that $(\hs{h}_1,\hs{t}_1)\ldots,(\hs{h}_n,\hs{t}_n)$ is a ladder
	and in particular separate $\CCv{x}$ and $\CCv{m}$. 
	This implies that they can be considered as halfspaces in $\Int{I}_1$ as well.
	Let $\CCc{D}_i$ be an \intc of maximal dimension for $(\hs{h}_i,\hs{t}_i)$ in the interval $\Int{I}_1$.
	
	Apply Lemma \ref{lemmaforlemma2}, for both $\CCc{C}_i$ and $\CCc{D}_i$. Denote by $\CCc{C}^b_i,\CCc{C}^{ub}_i$ (resp. $\CCc{D}^b_i,\CCc{D}^{ub}_i$) the \underline{b}ounded and \underline{u}n\underline{b}ounded halfspaces of $\CCc{C}_i$ (resp. $\CCc{D}_i$). Note that $\CCc{C}^b_i\cup \CCc{D}^{ub}_i$ (resp. $\CCc{C}^{ub}_i\cup\CCc{D}^{b}_i$) is an \intc for $(\hs{h}_i,\hs{t}_i)$ in $\Int{I}_1$ (resp. \intc for for $(\hs{h}_i,\hs{k}_i)$ in $\Int{I}_2$).
	Thus by assumption \[\#\CCc{C}^b_i+\#\CCc{C}^{ub}_i = \#\CCc{C}_i > \#\CCc{C}^{ub}_i + \#\CCc{D}^{b}_i.\] On the other hand, since $(\hs{h}_i,\hs{k}_i)$ is \adjP, \[\#\CCc{D}^b_i+\#\CCc{D}^{ub}_i=\#\CCc{D}_i \ge \#\CCc{C}^b_i + \#\CCc{D}^{ub}_i.\] Adding these two inequalities gives a contradiction.
	
	For the intervals $\Int{I}'_i$, the claim follows easily since if $(\hs{h},\hs{k})$ are locally parallel in $\Int{I}'_1$ then they must be locally parallel in $\Int{I}'_2$. This is because any halfspace which is greater than $\hs{h}$ in one of the intervals then it also belongs to the other interval.
\end{proof}
\section{Proof of the main theorem}
We follow the proof of Theorem A' in \cite{BeLa16}.

\begin{proof}[Proof of Theorem \ref{main result}]
Let $\usimp{K}$ be the universal cover of $\simp{K}$ and $\uptrn{P}$ the pattern on $\usimp{K}$ associated to $\ptrn{P}$. Since $\ptrn P$ is a $d$-pattern, the \CCC $\CC{X}$ is a $d$ dimensional cube complex.

For a vertex $\usimpv{x}$ in $\usimp{K}$ call $\CCv{\bar{x}}$ the corresponding vertex in $\CC{X}$. Similarly the halfspaces corresponding to a track $\utrk{t}$ in $\uptrn{P}$ are called $\hs{h}_{\utrk{t}}$ and $\comp{\hs{h}}_{\utrk{t}}$. A triangle in $\CC{X}$ is a triplets of vertices $(\CCv{\bar{x}}, \CCv{\bar{y}}, \CCv{\bar{z}})$  coming from a triangle $(\usimpv{x},\usimpv{y},\usimpv{z})$ of $\usimp{K}$.

Two tracks $\utrk{t}$ and $\utrk{t}'$ of $\uptrn P$ are \emph{\adjP}  if they cross an edge $[\usimpv{x},\usimp{y}]$ such that $\hs{h}_{\utrk{t}}$ and $\hs{h}_{\utrk{t}'}$ are \adjP in one of the oriented interval defined by $\CCv{\bar{x}}$ and $\CCv{\bar{y}}$.

Note that if two halfspaces $(\hs{h}, \hs{k})$ in $\CC{X}$ are not parallel but intersect an interval $[\CCv{\bar{v}}, \CCv{\bar{w}}]$ in which they are \adjP, then:
\begin{enumerate}
\item  either there exists some triangle $(\CCv{\bar{x}}, \CCv{\bar{y}}, \CCv{\bar{z}})$ such that $(\hs{h}, \hs{k})$ is \adjP in  $[\CCv{\bar{x}}, \CCv{\bar{y}}]$ but is separated by the midpoint of $(\CCv{\bar{x}}, \CCv{\bar{y}}, \CCv{\bar{z}})$,
\item or  there exists some triangle $(\CCv{\bar{x}}, \CCv{\bar{y}}, \CCv{\bar{z}})$ such that $(\hs{h}, \hs{k})$ is \adjP in  $[\CCv{\bar{x}}, \CCv{\bar{y}}]$, intersects $[\CCv{\bar{x}}, \CCv{\bar{z}}]$ but is not \adjP in it.
\end{enumerate}

If there are no parallel tracks in $\ptrn P$, a halfspace $\hs h$ in $\CC{X}$  belongs to one of the following categories that can be bounded.

\begin{enumerate}
\item \label{enum1} The halfspace $\hs h$ is associated to a track belonging to an edge of $K$ which is not in a triangle. Two tracks of this form on the same edge are parallel, therefore on each edge there is at most one track $\trk{t}$, associated to two halfspaces $\hs{h}_{\trk{t}}$ and $\comp{\hs{h}}_{\trk{t}}$.
\item \label{enum2} The halfspace $\hs h$ belongs to an interval $[\CCv{\bar{x}}, \CCv{\bar{y}}]$ and is maximal in it. For each directed interval there are at most $d$ maximal halfspaces, and thus at most $2 d$ per edge.  Note that this case contains the previous one.
\item  \label{enum3} There exist some halfspace $\hs k$ and some triangle $(\CCv{\bar{x}}, \CCv{\bar{y}}, \CCv{\bar{z}})$ such that $(\hs{h}, \hs{k})$ is \adjP in  $[\CCv{\bar{x}}, \CCv{\bar{y}}]$ but is separated by the midpoint of $(\CCv{\bar{x}}, \CCv{\bar{y}}, \CCv{\bar{z}})$. By lemma \ref{lemma1} each triangle and directed interval $\Int{I}$ defined by an edge of the triangle, there is a bound $C_1$ of pairs of \adjP halfspaces in $\Int{I}$ separated by the midpoint of the triangle. There are $6$ directed intervals associated to each triangle.
\item \label{enum4} There exist some halfspace $\hs k$ and some triangle $(\CCv{\bar{x}}, \CCv{\bar{y}}, \CCv{\bar{z}})$ such that $(\hs{h},\hs{k})$ is \adjP in  $[\CCv{\bar{x}}, \CCv{\bar{y}}]$, intersects $[\CCv{\bar{x}}, \CCv{\bar{z}}]$ but is not \adjP in it. By lemma \ref{lemma2}, for each triangle and each pair of intervals $[\CCv{\bar{x}}, \CCv{\bar{y}}]$ and $[\CCv{\bar{x}}, \CCv{\bar{z}}]$ there is a bound $C_2$ of pairs of halfspaces   that intersect both $[\CCv{\bar{x}}, \CCv{\bar{y}}]$ and $[\CCv{\bar{x}}, \CCv{\bar{z}}]$, \adjP in the first one but not the second one. There are $6$ choices of such a pair of intervals per triangle.
\item \label{enum5} There exist some halfspace $\hs k$ and some triangle $(\CCv{\bar{x}}, \CCv{\bar{y}}, \CCv{\bar{z}})$ such that $(\hs{h},\hs{k})$ is \adjP in  $[\CCv{\bar{y}}, \CCv{\bar{x}}]$, intersects $[\CCv{\bar{z}}, \CCv{\bar{x}}]$ but is not \adjP in it. By lemma \ref{lemma2}, given a triangle and a pair of intervals $[\CCv{\bar{x}}, \CCv{\bar{y}}]$ and $[\CCv{\bar{x}}, \CCv{\bar{z}}]$ there are no pair of halfspaces  that intersects both $[\CCv{\bar{y}}, \CCv{\bar{x}}]$ and $[\CCv{\bar{y}}, \CCv{\bar{x}}]$, \adjP in the first one but not the second one.
\end{enumerate}

 If we denote by $E$ and $T$ the number of edges and triangles in $\simp K$, then there are at most $2dE +(6C_1+6C_2)T$ non parallel halfspaces  in $\CC{X}$.
\end{proof}

\section{Cubical acylindricity}



Using Theorem \ref{main result} and following the proof of Theorem 1 in \cite{Del99}, we prove Theorem \ref{acylindrical accessibility}.

\begin{proof}[Proof of Theorem \ref{acylindrical accessibility}]
	Let $\simp{K}$ be a presentation complex for $G$, so that $\pi_1 (\simp{K})=G$.
	Let $\CC{X}$ be a $d$-dimensional CAT(0) cube complex on which $G$ acts $(k,\mathcal{C})$-acylindrically on hyperplanes. 
	Pullback the hyperplanes of $\CC{X}$ to get a $d$-pattern $\ptrn{P}$ on $\simp{K}$ (see construction in Section \ref{resolutions}).
	Every hyperplane of $\CC{X}$ has at least one track in its pullback which is $G$-essential in the induced CAT(0) cube complex.
	Remove all non-$G$-essential tracks from the pattern.
	
	Let $R=R(k,d)$ be as in Observation \ref{Ramsey}, and let $C=C(K,d)$ be as in Theorem \ref{main result}.
	By the pigeon hole principle, if $\ptrn{P}$ has more than $R\cdot C$ tracks, then there are $R$ tracks which belong to the same parallelism class, and hence $k$ of them correspond to a chain in $\CC{X}$. 
	Let $\trk{t}$ be a track in this parallelism class.
	
	Since any element that stabilizes the hyperplane defined by $\trk{t}$ also stabilizes the set of tracks in the parallelism class of $\trk{t}$. 
	Thus, up to passing to a finite index subgroup it stabilizes each of the tracks in the parallelism class, and hence in the common stabilizer of the corresponding hyperplanes in $\CC{X}$.
	By the $(k,\mathcal{C})$-acylindricity on hyperplanes of the action, the stabilizer of $\trk{t}$ is in $\mathcal {C}$ since it stabilizes a chain of $k$ hyperplanes in $\CC{X}$. 
	The hyperplane defined by this track alone gives a $d$-pattern on $G$, which, by Proposition 3.2 of \cite{CaSa11} induces an essential $G$-action on a $d$-dimensional CAT(0) cube complex whose hyperplane stabilizers are in $\mathcal {C}$.
	Contradicting the assumption on $G$.
\end{proof}

\begin{proposition}\label{cube complexes to trees}
	Let $G$ be a finitely presented group. 
	\begin{enumerate}
	\item \label{one end implies no CCC over finite} If $G$ acts on finite dimensional \CCC with finite hyperplane stabilizers. Then either $G$ fixes a point or has more than one end.
	\item \label{Z CCC implies Z tree}
	If $G$ is moreover one-ended hyperbolic group and is not a triangle group. If $G$ acts on finite dimensional \CCC with virtually cyclic hyperplane stabilizers, then either $G$ fixes a point or $G$ splits over a cyclic group.
	\end{enumerate}
\end{proposition}

\begin{proof}
	Let $K$ be a the presentation complex of $G$.
	Let $\uptrn{P}$ be the pattern obtained by a pullback of the hyperplanes of the \CCC on which $G$ acts, and let $\CC{X}'$ be the induced cube complex.
	There are only finitely many orbits of hyperplanes in $\CC{X}'$.
	By Proposition 3.5 in \cite{CaSa11}, we may assume that the action is also essential by removing the non-essential tracks. 
	As always for finitely presented, $G$ acts cocompactly on the tracks of the pattern $\ptrn{P}$. 
	
	In the setting of \ref{one end implies no CCC over finite}, the tracks are essential and finite, proving that $G$ has more than one end.

	To prove \ref{Z CCC implies Z tree}, note that 
	by \ref{one end implies no CCC over finite}, either $G$ fixes a vertex of the resolution, and hence in the original action, or the track stabilizers are infinite virtually cyclic subgroups.
	In this case, since each track separates $\uc{K}$ to two essential components, and any virtually cyclic group is quasiconvex, we obtain a separating pair of points at the boundary.
	By Theorem 6.2 of \cite{Bow98}, this implies that $G$ splits over a virtually cyclic group.
\end{proof}

We finish this section by showing that acylindrical on hyperplanes actions on cube complexes and hyperbolic cubulations are the same for geometric actions.

\begin{proposition}\label{hyperbolicity acylindricity}
	Let $G$ be a group acting properly, cocompactly  on a CAT(0) cube complex $\CC{X}$. Then, $G$ is hyperbolic if and only if $G$ acts $(k,\mathcal F)$-acylindrically on hyperplanes, for some $k\in\N$.
\end{proposition}

\begin{proof}
	If $G$ is hyperbolic then the cube complex $\CC{X}$ is $\delta$-hyperbolic for some $\delta$, and hence if it is not $(k,\mathcal{F})$-acylindrically on hyperplanes for any $k\in\N$ then one can find an arbitrarily wide strip in $\CC{X}$, contradicting hyperbolicity.
	
	For the converse, by the Corollary of \cite{Bri95}, it suffices to show that there are no flats in $\CC{X}$. Assume $F$ is a 2-dimensional flat in $\CC{X}$. Let $\Hs{H}_F$ be the hyperplanes that are transverse to the flat $F$. There is a chain of hyperplanes in $\Hs{H}_F$ of length $k$ which intersect $F$ in parallel lines. This implies that there are two hyperplanes $\hyp{h},\hyp{k}$ whose common stabilizer is finite but their $R$ neighborhoods have unbounded intersection for some $R>0$. By a standard argument this implies that the common stabilizer is infinite, contradicting the acylindricity on hyperplanes.
\end{proof}

%

%

\bibliographystyle{hplain}
\bibliography{main}

\end{document}